\documentclass[reqno,english]{amsart}
\usepackage{etex}
\usepackage{amsmath,amssymb,amsthm,bbm,mathtools,comment}
\usepackage[shortlabels]{enumitem}
\usepackage[pdftex,colorlinks,backref=page,citecolor=blue]{hyperref}
\usepackage[mathscr]{euscript}
\usepackage{tikz,babel,adjustbox}
\usepackage{cmtiup}
\usepackage{amsfonts}
\usepackage{graphicx}
\usepackage{caption}
\usepackage{verbatim}
\usepackage{array}
\usepackage[frame,cmtip,arrow,matrix,line,graph,curve]{xy}
\usepackage{graphpap, color, pstricks}
\usepackage{pifont}
\usepackage{cmtiup}
\usepackage{amssymb}
\usepackage{amsthm}
\allowdisplaybreaks

\theoremstyle{plain}
\newtheorem{theorem}{Theorem}[section]

\newtheorem{lem}[theorem]{Lemma}

\newtheorem{proposition}[theorem]{Proposition}
\newtheorem{prop}[theorem]{Proposition}

\theoremstyle{remark}

\newcommand{\beq}[1]{\begin{equation}\label{#1}}
\newcommand{\enq}[0]{\end{equation}}

\usepackage{etex}
\usepackage{amsmath,amssymb,amsthm,bbm,mathtools,comment}
\usepackage[shortlabels]{enumitem}
\usepackage[pdftex,colorlinks,backref=page,citecolor=blue]{hyperref}
\usepackage[mathscr]{euscript}
\usepackage{tikz,babel,adjustbox}
\usepackage{cmtiup}
\usepackage{amsfonts}
\usepackage{graphicx}
\usepackage{caption}
\usepackage{verbatim}
\usepackage{array}
\usepackage[frame,cmtip,arrow,matrix,line,graph,curve]{xy}
\usepackage{graphpap, color, pstricks}
\usepackage{pifont}
\usepackage{cmtiup}
\usepackage{amssymb}

\allowdisplaybreaks

\theoremstyle{plain}

\title{\textbf{Characterization of Maximal Left-Compressed Intersecting Families Generated by a Single Generator}}

\author{\raggedright \textnormal {Nguyen Trong Tuan }}
\email{23n21103@student.hcmus.edu.vn}
\address{Faculty of Mathematics and Computer Science, University of Science, Vietnam National University - Ho Chi Minh City, Vietnam}

\begin{document}
	
	\begin{abstract} 
		In this paper, we focus on left-compressed families \(\mathcal{F}(\mathcal{G})\) generated by a family of sets \(\mathcal{G}\). We establish a necessary and sufficient condition on \(\mathcal{G}\) under which \(\mathcal{F}(\mathcal{G})\) is intersecting. We then introduce a construction of maximal left-compressed intersecting families \(\mathcal{F}_G\) from a single generator by adding suitable sets not containing \(1\) to \(F(G)\). From this, we derive several estimates for the size of \(\mathcal{F}_G\).
	\end{abstract}
	\maketitle
	\section{Introduction}
	\subsection{Basic Terminology and Notation}
	Throughout this paper, we assume that \(n\) and \(k\) are positive integers satisfying \(4 \le 2k \le n\).
	For each positive integer $n$, we denote by $[n]$ the set $\{1,2,\ldots,n\}$, and for positive integers \(m<n\), we write $[m,n]=\{m,m+1,\ldots,n\}$. A subset $A$ of $[n]$ with $r$ elements is called an $r$-set, and its elements are always listed in increasing order. 
	
	For $A=\{a_1,a_2,\ldots,a_r\}$, with $1\le a_1<a_2<\cdots<a_r\le n$, we define $\min A=a_1$ and $\max A=a_r$. For $1 \le s \le r$, we denote by $\pi_s(A)$ the set $\{a_1,\ldots,a_s\}$. If $\min A > m$, we define $A-m=\{a_1-m,\ldots,a_r-m\}$. Similarly, if $\max A < m$, we define
	$m-A=\{m-a_r,\ldots,m-a_1\}$.
	
	We denote by \(\binom{[n]}{k}\) the family of all \(k\)-subsets of \([n]\). A family \(\mathcal{A} \subseteq \binom{[n]}{k}\) is called a \(k\)-uniform family. For any family \(\mathcal{A} \subseteq 2^{[n]}\), we say that \(\mathcal{A}\) is \textit{intersecting} if \(A \cap B \ne \varnothing\) for all \(A,B \in \mathcal{A}\). Two families \(\mathcal{A},\mathcal{B}\) are called \textit{cross-intersecting} if for every \(A \in \mathcal{A}\) and for every  \(B \in \mathcal{B}\), we have \(A \cap B \ne \varnothing\).
	
	We define a partial order on the \(r\)-subsets, where \(r \le k\), as follows. For two sets \(A,B \in \binom{[n]}{r}\) with \(A=\{a_1,a_2,\ldots,a_r\}\) and \(B=\{b_1,b_2,\ldots,b_r\}\), we define \(A \le B\) if \(a_i \le b_i\) for all \(i=1,2,\ldots,r\). We adopt the  convention that $\emptyset \le A$ for all $A$. We extend this order to the whole power set \(2^{[n]}\) as follows. For two sets \(A=\{a_1,a_2,\ldots,a_r\}\) and \(B=\{b_1,b_2,\ldots,b_s\}\), we define \(A \preceq B\) if \(r \ge s\) and \(a_i \le b_i\) for all \(1 \le i \le s\). Equivalently, $A \preceq B$ if the first $|B|$ elements of $A$ do not exceed the corresponding elements of $B$, that is $\pi_{|B|}(A) \le B$. Note that if \(B \subseteq A\), then  \(A \preceq B\).
	
	A family \(\mathcal{A} \subseteq 2^{[n]}\) is \textit{left-compressed} if for every \(A \in \mathcal{A}\) and every \(B \le A\), we have \(B \in \mathcal{A}\). We say that a \(k\)-set \(A\) is a \textit{maximal element} of a left-compressed intersecting family \(\mathcal{A} \subseteq \binom{[n]}{k}\) if there is no \(B \in \mathcal{A}\) such that \(A \le B\) and $A \neq B$. Thus, if \(\mathcal{A}\) is left-compressed, it can be described completely by listing its maximal elements with respect to this order. A natural example of a left-compressed family is $\mathcal{L}(A)=\{S \in \binom{[n]}{|A|}: S \le A\}$ for $A \in 2^{[n]}$. We also define $\mathcal{L}_{0}(A)=\{S \in \binom{[2,n]}{|A|}: S \le A\}$. 
	
	A family that is both intersecting and left-compressed is called a \textit{left-compressed intersecting family}; for short, we call it an LCIF. The following are some familiar examples of LCIFs
	: The family \textit{star} $S=\{ A\in \binom{[n]}{k}: 1\in A\}$; the family $\mathcal{A}_{2,3}=\{A \in \binom{[n]}{k}:|A\cap \{1,2,3\}|\geq 2\}$; the family \textit{Hilton-Milner}  $\mathcal{H}\mathcal{M}=\{A\in \mathcal{S}:A\cap \{1,2,3\} \neq \emptyset \}$. 
	
	Finally, an intersecting family $\mathcal{A} \subseteq \binom{[n]}{k}$ is called maximal if no $k$-set can be added to $\mathcal{A}$ without destroying the intersecting property. 
	
	\subsection{Left-Compressed Families Generated by Families of Sets}
	
	The main objects studied in this paper are left-compressed families generated by collections of sets. Such families were introduced by Ahlswede and Khachatrian \cite{AK}, and later a variant suitable for the study of left-compressed families was considered by Baber \cite{B}. Let $\mathcal{G}\subseteq2^{[n]}$. The $k$-uniform family on $[n]$ generated by $\mathcal{G}$ is defined by
	
	\[
	\mathcal{F}(n,k,\mathcal{G})
	=
	\left\{
	A\in\binom{[n]}{k}:
	A\preceq G
	\text{ for some }
	G\in\mathcal{G}
	\right\}.
	\]
	
	Throughout the paper, whenever $n$ and $k$ are fixed, we simply write
	$\mathcal{F}(\mathcal{G})$ instead of
	$\mathcal{F}(n,k,\mathcal{G})$. We call \(\mathcal{G}\) a \emph{generating family} and each set \(G \in \mathcal{G}\) a \emph{generating set}. If $G\in\mathcal{G}$ satisfies $|G|>k$, then there is no $k$-set
	$A$ with $A\preceq G$. Therefore, throughout the paper we may assume
	that every generating set has size at most $k$. In the special case where $\mathcal{G}=\{G\}$ consists of a single set, we simply write $\mathcal{F}(G)$. If $|G|=k$, then
	$\mathcal{F}(G)=\mathcal{L}(G)$.
	If $|G|<k$, then $\mathcal{F}(G)
	=\mathcal{L} \bigl( G\cup[n-k+|G|+1,n]
	\bigr).$. We also define $\mathcal{F}_0(G)$ to be the subfamily of $\mathcal{F}(G)$ consisting of all sets that do not contain the
	element $1$.
	
	By definition,
	
	\[
	\mathcal{F}(\mathcal{G})
	=
	\bigcup_{G\in\mathcal{G}}
	\mathcal{F}(G),
	\]
	
	and hence every family of the form
	$\mathcal{F}(\mathcal{G})$
	is left-compressed.
	
	However, determining whether
	$\mathcal{F}(\mathcal{G})$
	is intersecting is considerably more difficult.
	For example, $\mathcal{F}(\{\{2,3\}\})$ and $\mathcal{F}(\{\{1,3\},\{1,4,5\},\{2,3,5\}\}) $
	are intersecting, whereas $\mathcal{F}(\{\{2,4\}\})$ and $\mathcal{F}(\{\{2,3\},\{2,4,5\}\})$ are not.
	
	Bond~\cite{BB} obtained a necessary and sufficient condition for a
	left-compressed family to be intersecting.
	
	\medskip
	
	\noindent
	\textbf{Bond~\cite[Proposition~7]{BB}.}
	A left-compressed family
	$\mathcal{A}\subseteq\binom{[n]}{k}$
	is intersecting if and only if, for every $A=\{a_1,\ldots,a_k\}, B=\{b_1,\ldots,b_k\}
	\in\mathcal{A}$, there exist indices $i$ and $j$ such that $i+j>\max\{a_i,b_j\}$
	
	Bond's criterion completely characterizes intersecting left-compressed
	families. However, it is formulated in terms of the generated family
	itself. Since $\mathcal{F}(\mathcal{G})$ may contain a very large
	number of $k$-sets, applying this criterion may require examining many
	members of the generated family. This naturally raises the following
	question.
	
	\begin{center}
		\emph{Can one determine whether
			$\mathcal{F}(\mathcal{G})$
			is intersecting directly from the generating family
			$\mathcal{G}$?}
	\end{center}
	
	The first objective of this paper is to answer this question. To this
	end, we introduce the following notion. For two nonempty sets
	$G,H\subseteq[n]$
	(not necessarily distinct),
	we say that $G$ and $H$ are \emph{strongly intersecting}, written
	$G\sim_{si}H$, if $G'\cap H'\neq\varnothing$ for every $G'\le G$
	and every $H'\le H$. When $G=H$, we say that $G$ is \emph{self-strongly intersecting}. A family
	$\mathcal{G}\subseteq2^{[n]}$
	is called \emph{strongly intersecting}
	if every pair of its members is strongly intersecting.
	
	Our first theorem completely characterizes when a generated
	left-compressed family is intersecting.
	
	\begin{theorem}\label{thm1}
		Let $\mathcal{G}\subseteq2^{[n]}$.
		Then
		$\mathcal{F}(\mathcal{G})$
		is intersecting if and only if
		$\mathcal{G}$
		is strongly intersecting.
	\end{theorem}
	The characterization in Theorem~\ref{thm1} first appeared in the unpublished preprint~\cite{TTT}. We include it here because it provides the key tool for the results established in the remainder of this paper.
	
	Theorem \ref{thm1} states that if \(\mathcal{F}(\mathcal{G})\) is intersecting, then every set \(G\in\mathcal{G}\) is a self-strongly intersecting set. As will be shown in Lemma \ref{main_lem}, if \(G=\{g_1,\ldots,g_r\}\) is an self-strongly intersecting, then there exists $1 \le s \le r$ such that \(g_s\le 2s-1\). The smallest index \(s\) for which \(g_s \le 2s-1\) is called the type of \(G\) and is denoted by \(t(G)\). Clearly, $t(G) \le |G|$. If $t(G)=s >1$ then it is easy to see that $g_i \ge 2i$ for all $1 \le i \le s=2, g_{s-1}=2s-2$ and $g_{s}=2x-1$. 
	
	We now turn to maximal left-compressed intersecting families (MLCIFs). A natural question is how to construct a maximal left-compressed intersecting family. Baber~\cite{B} showed that every MLCIF on $[n]$ can be obtained by
	extending one on $[2k]$.
	
	\medskip
	
	\noindent
	\textbf{Baber~\cite[Lemma~8]{BB}.}
	Let
	$\mathcal{A}\subseteq\binom{[2k]}{k}$
	be a maximal left-compressed intersecting family.
	Then every $n\ge2k$admits a unique maximal left-compressed intersecting extension of $\mathcal{A}$.Moreover, every maximal left compressed intersecting family on $\binom{[n]}{k}$
	arises in this way.
	
	Our approach is different. Instead of extending an MLCIF on $[2k]$, we
	construct maximal left-compressed intersecting families directly from a
	single self-strongly intersecting set satisfying a special property.
	
	We call a self-strongly intersecting
	$G$ a \emph{generator} if $1\notin G$
	and $|G\cap[k+|G|-1]|=|G|$. A generator \(G\) is called a weak generator if \(t(G)<|G|\), and a strong generator if \(t(G)=|G|\). Note that if \(G\) is a weak generator of type \(s\), then \(G\setminus\{\max G\}\) is also a weak generator of type \(s\).
	
	By Theorem 1.1, every generator \(G\) gives rise to an intersecting family \(\mathcal{F}(G)\). However, \(\mathcal{F}(G)\) is, in general, not maximal. This naturally leads to the problem of extending \(\mathcal{F}(G)\) to a maximal left-compressed intersecting family. We address this problem by introducing a canonical extension \(\mathcal{F}_G\), which forms the basis for the subsequent results of the paper. To begin with, we introduce the following notion. Let $G=\{g_1,\ldots,g_r\}$ be a genererator. For each $1\le i\le r$, define $G_i=\{1\}\cup[i+1,g_i]$. Since
	$i+(g_i-i+1)-1=g_i$,
	Lemma~\ref{main_lem} implies that $G\sim_{si}G_i$ for every $i$. We denote $\mathcal{H}(G)=\{G_1,\ldots,G_r\}$. Note that $\{G\} \cup \mathcal{H}(G)$ is strongly intersecting. So the family $\mathcal{F}(G)\cup \mathcal{F}(\mathcal{H}(G))=\mathcal{F}(G)\cup \mathcal{F}(G_1)\cup \ldots\cup \mathcal{F}(G_r)$ is intersecting. We define $F_G=\mathcal{F}(G)\cup \mathcal{F}(\mathcal{H}(G)$. We are now ready to state our main theorem.
	
	\begin{theorem} \label{thm2}
		Let $G=\{g_1,\ldots,g_r\} \in 2^{[n]}$ be a generator. Then,
		
		(i) the family $\mathcal{F}_{G}=\mathcal{F}(G) \cup \mathcal{F}(\mathcal{H}(G))$ is an MLCIF and
		\[
		|\mathcal{F}_{G}|
		=
		\sum_{A \in \mathcal{L}_{0}(G)}
		\binom{n-\max A}{k-|G|}
		+
		\binom{n-1}{k-1}
		-
		|\mathcal{L}(\pi_{k-1}(n+1-\overline{G}))|.
		\]
		
		(ii) If $G$ is a weak generator then
		\begin{itemize}
			\item If $g_r<2r-2$ then $|\mathcal{F}_G| > |\mathcal{F}_{G\setminus \{g_r\}}|$
			
			\item If $g_r=2r-2$ then $|\mathcal{F}_G| = |\mathcal{F}_{G\setminus \{g_r\}}|$
			
			\item If $g_r>2r-2$ then $|\mathcal{F}_G| <|\mathcal{F}_{G\setminus \{g_r\}}|$
		\end{itemize}
		
		(iii) If $G$ is a strong generator then for every generator $K \le G$, we have $|\mathcal{F}_K| \ge |\mathcal{F}_G|$.
	\end{theorem}

	Our paper is divided into three parts. The first part provides an overview of the fundamental concepts and properties, as well as outlines the structure of the paper. In the second section, we discuss LCIFs. More specifically, we prove Theorem \ref{thm1} which establishes the necessary and sufficient conditions for the family $\mathcal{F}(\mathcal{G})$ to be an LCIF. In addition, we provide an alternative proof of Bond's condition \cite{BB} that does not rely on Frankl's result \cite{Fr87}. In the third section, we present the proof of Theorem \ref{thm2} concerning the construction of the MLCIF \(F_G\) and several estimates related to its size.
	
	\section{Left Compressed Intersecting Families }
	\subsection{Preliminaries}
	We proceed to establish several supporting properties that will be used in the proof of Theorem \ref{thm1}. We need a function that describes the relationship between a set $X$ and an element $\ell$. For a non-empty set $X$ and a positive integer $\ell$, we define $\mu_{X}(\ell)$ as the number of elements in $X$ that are less than or equal to $\ell$. In other words, $\mu_{X}(\ell)=|X \cap [\ell]|$. We note that for any $X \subseteq [n]$, we have $0 \leq \mu_{X}(\ell) \leq |X|,$ and for fixed $X$, $\mu_{X}(\ell)$ is a non-decreasing function in $\ell$. For $\ell \geq 2$, we also have $\mu_{X}(\ell) - \mu_{X}(\ell-1) \in \{0,1\}$. Next, we establish some additional properties of the function $\mu_{X}(\ell)$ that relate to two sets $A,B \in 2^{[n]}$.
	
	\begin{prop} \label{tc} Let $A,B \subseteq [n]$ be non-empty subsets. Then, for all $\ell \in [n]$, the following hold: \\
		(i) If $A \preceq B$ then $\mu_{A} (\ell) \geq \mu_{B}(\ell)$ ;\\
		(ii) If $A \leq B$ then $\mu_{A}(\ell) \geq \mu_{B}(\ell)$;\\
		(iii) If $A \subseteq B $ then $\mu_{B}(\ell) \geq \mu_{A}(\ell)$.
	\end{prop}
	\begin{proof}
		(i) Assume that $A=\{a_{1},a_{2},\ldots,a_{r}\}$ and $B=\{b_{1},b_{2},\ldots,b_{
			s}\}$. Since $A \preceq B$, we have $r \geq s$ and\\ $\{a_{1},a_{2},\ldots,a_{s}\} \leq \{b_{1},b_{2},\ldots,b_{s}\}$.
		Now, we consider different cases for $\ell \in [n]$. If $\ell \geq b_{s} $ then $\mu_{B}(\ell)=s $. Since $a_{s} \leq b_{s} \leq \ell$, we also have $\mu_{A}(l) \geq s =\mu_{B}(l)$. If $\ell < b_{1}$ then $\mu_{B}(\ell) =0$, which proves the required result. Finally, consider $b_{1} \leq \ell < b_{s}$. Assume $|B \cap [\ell]|=p$. This means that there are p elements $b_{1},\ldots,b_{p}$ that do not exceed $\ell$. Since $a_{1} \leq b_{1},\ldots,a_{p} \leq b_{p} \leq \ell$, we conclude that there are at least $p$ numbers $a_{1},\ldots,a_{p}$ that do not exceed $\ell$. Thus $\mu_{A}(\ell) \geq p =\mu_{B}(\ell)$.
		
		(ii) Since  $A\leq B$, we conclude that $A \preceq B$ and from (i), we have $\mu_{A}(\ell) \geq \mu_{B}(\ell)$.
		
		(iii) Since $A \subseteq B$, we have $B \preceq A$. From (i), we get $\mu_{B}(\ell) \geq \mu_{A}(\ell)$ for all $1 \leq \ell \leq n$.
	\end{proof}
	
	The following lemma plays an important role in the proof of Theorem \ref{thm1} and also concerns the strongly intersecting property of two sets.
	
	\begin{lem}\label{main_lem} Let $G,H \in 2^{[n]}$. Then, the following statements are equivalent:
		
		(i) $G \sim_{si} H$;
		
		(ii) $\mathcal{F}(G)$ and $\mathcal{F}(H)$ are cross-intersecting;
		
		(iii) There exists $1 \leq \ell \leq n$ such that $\mu_{G}(\ell)+\mu_{H}(\ell) > \ell$;
		
		(iv) There exist $1 \le p \le |G|$ and $1 \le q \le |H|$ such that $g_p=h_q=p+q-1$.
	\end{lem}
	\begin{proof} Assume that $G=\{g_1,\ldots,g_r\}$ and $H=\{h_1,\ldots,h_s\}$.
		
		$(i) \Rightarrow (ii)$
		
		Take $A \in \mathcal{F}(G), B \in \mathcal{F}(H)$. Then, $A \preceq G$ and $B \preceq H$. Let $A'$ be the set of the first $|G|$ elements of $A$ and $B'$ be the set of the first $|H|$ elements of $B$. Then, $A' \leq G, B' \leq H$. Since $G \sim_{si} H$, we get $A' \cap B' \neq \emptyset$ and therefore $A \cap B \neq \emptyset$.
		Thus $\mathcal{F}(G)$ and $\mathcal{F}(H)$ are cross-intersecting.\\ \\
		$(ii) \Rightarrow (iii)$\\
		Suppose that $\mathcal{F}(G)$ and $\mathcal{F}(H)$ are cross-intersecting. For simplicity, for each $ 1 \le \ell \le n$, we denote $x_{\ell}=\mu_{G}(\ell), y_{\ell}=\mu_{H}(\ell), z_{\ell}=x_{\ell}+y_{\ell}$. We have the following simple observations. We have $x_{\ell+1}=x_{\ell}+1$ if $\ell+1 \in G$ and $x_{\ell+1}=x_{\ell}$ if $\ell+1 \notin G$. Similarly, we also have $y_{\ell+1}=y_{\ell}+1$ if $\ell+1 \in H$ and $y_{\ell+1}=y_{\ell}$ if $\ell+1 \notin H$. Therefore, $z_{\ell+1}-z_{\ell} \in \{0,1,2\}$.\\
		Assume, for a contradiction, that $\mu_{G}(\ell)+\mu_{H}(\ell) \leq \ell$ for all $1 \leq \ell \leq n$. Thus, $z_{l} \leq \ell$, for all $1 \leq \ell \leq n$. By induction, for each $\ell$, we will construct the sets $G_{\ell}$ and $H_{\ell}$ satisfying simultaneously the following conditions: $G_{\ell} \leq G \cap [\ell], H_{\ell} \leq H \cap [\ell], G_{\ell} \cap H_{\ell}= \emptyset$ and $G_{\ell} \cup H_{\ell}=[z_{\ell}]$.\\ 
		
		Consider $\ell=1$. We have $z_{1} \leq 1$. So $z_{1}=0$ or $z_{\ell}=1$. If $z_{\ell}=0$ then we take $G_{1}=H_{1}= \emptyset$ and we see that the above conditions are trivially satisfied. We consider the case where $z_{1}=x_{1}+y_{1}=1$. We may assume that $x_{1}=1$ and $y_{1}=0$. This means that $1 \in G$ and $1 \notin H$. We take $G_{1}=\{1\} \leq G \cap [1]$ and $H_{1}= \emptyset \leq H \cap [1]$. The above conditions are also satisfied.\\
		Assume that we have constructed two sets $G_{\ell}$ and $H_{\ell}$ that satisfy the conditions $G_{\ell} \leq G \cap [\ell], H_{\ell} \leq H \cap [\ell], G_{\ell} \cap H_{\ell}= \emptyset$ and $G_{\ell}\cup H_{\ell}=[z_{\ell}]$. We will build the sets $G_{\ell+1}$ and $H_{\ell+1}$. There are three possible cases. 
		
		\textbf{ The case of $z_{\ell+1}=z_{\ell}$}\\
		We have $\ell+1 \notin G\cup H$. We set $G_{\ell+1}=G_{\ell}$ and $H_{\ell+1}=H_{\ell}$. We have, by the induction hypothesis, $G_{\ell+1}=G_{\ell} \leq G \cap [\ell]= G \cap [\ell+1]$ and $H_{\ell+1}=H_{\ell} \leq H \cap [\ell]= H \cap [\ell+1]$. Furthermore, $G_{\ell+1} \cap H_{\ell+1}= G_{\ell} \cap H_{\ell}= \emptyset$ and $G_{\ell+1} \cup H_{\ell+1}=G_{\ell} \cup H_{\ell}=[z_{\ell}]=[z_{\ell+1}]$. 
		
		\textbf{ The case of $z_{\ell+1}=z_{\ell}+1$}\\
		We have $x_{\ell+1}=x_{\ell}+1, y_{\ell+1}=y_{\ell}$ or $y_{\ell+1}=y_{\ell}+1, x_{\ell+1}=x_{\ell}$. Without loss of generality, we can see that $x_{\ell+1}=x_{\ell}+1, y_{\ell+1}=y_{\ell}$. Thus, $\ell+1 \in G$ and $\ell+1 \notin H$. We set $G_{\ell+1}=G_{\ell} \cup \{z_{\ell+1}\}$ and $H_{\ell+1}=H_{\ell}$. According to the induction hypothesis, $G_{\ell} \leq G \cap [\ell]$ and $H_{\ell} \leq H \cap [\ell]$. Note that $z_{\ell+1} \leq \ell+1$. It implies $G_{\ell+1}=G_{\ell} \cup \{z_{\ell+1}\} \leq G\cap [\ell+1]$ and $H_{\ell+1}=H_{\ell} \leq H\cap [\ell]=H \cap [\ell+1]$.   Also, by the induction hypothesis, we have $G_{\ell+1} \cap H _{\ell+1}= (G_{\ell} \cup \{z_{\ell+1}\}) \cap H_{\ell}=(G_{\ell} \cap H_{\ell})\cup (H_{\ell} \cap \{z_{\ell+1}\})=\emptyset$ since $z_{\ell+1} \notin H_{l}$. We also have $G_{l+1} \cup H_{\ell+1}=(G_{\ell} \cup H_{\ell}) \cup \{z_{\ell+1}\}=[z_{\ell}] \cup \{z_{\ell+1}\}=[z_{\ell+1}]$.
		
		\textbf{ The case of $z_{\ell+1}=z_{\ell}+2$}\\
		In this case, we have $x_{\ell+1}=x_{\ell}+1$ and $ y_{\ell+1}=y_{\ell}+1$. So $\ell+1 \in G \cap H$. We set $G_{\ell+1} =G_{\ell} \cup \{z_{\ell}+1\}$ and $H_{\ell+1} =H_{\ell} \cup \{z_{\ell}+2\}$. We check all the conditions. First, $G_{\ell+1} \cup H_{\ell+1} = (G_{\ell} \cup \{z_{\ell+1}-1\}) \cup (H_{\ell} \cup \{z_{\ell+1}\})= (G_{\ell} \cup H_{\ell}) \cup \{z_{\ell+1}-1\}  \cup \{z_{\ell+1}\}=[z_{\ell}]\cup \{z_{\ell+1}-1\} \cup\{z_{\ell+1}\}=[z_{\ell+1}]$. Next, since $G_{\ell} \cap H_{\ell}= \emptyset$, we have $G_{\ell+1} \cap H_{\ell+1}= (G_{\ell} \cup \{z_{\ell}+1\}) \cap (H_{\ell} \cup \{z_{\ell}+2\})= \emptyset$. Finally, since $\ell+1 \in G \cap H$ and $G_{\ell} \leq G \cap [\ell]$, we have $G_{\ell+1}=G_{\ell} \cup \{z_{\ell}+1\} \leq (G \cap [\ell]) \cup \{\ell+1\}=G \cap [\ell+1]$. Similarly, we also have $H_{\ell+1} \leq H \cap [\ell+1]$.
		
		Now, we let $m = \max \{g_{r}, h_{s}\}$. Since $G \cap [m] = G$ and $H \cap [m]=H$, we have $x_m=\mu_G(m)=|G|$ and $ y_m=\mu_H(m)=|H|$. Applying the above construction for $\ell=m$, we get two sets $G_{m}, H_{m}$ satisfying the conditions: $G_{m} \leq G\cap[m]=G, H_{m} \leq H\cap[m]=H, G_{m} \cup H_{m}=[z_m]$ and $G_{m} \cap H_{m}=\emptyset$. We also have $z_m=x_m+y_m=|G|+|H|$. Next, we take two sets $K,J \subseteq[z_m+1,n]$ such that $|K|=k-|G|, |J|=k-|H|$ and $K \cap J=\emptyset$ (it is possible to choose two sets $K$ and $J$ because $(k-|G|)+(k-|H|)=2k -(|G|+|H|)=2k-z_m  \leq n-z_m )$. We define $A=G_{m} \cup K$ and $B=H_{m} \cup J$. We have $|A|=|B|=k$, $A \preceq G$ and $B \preceq H$. Thus, $A \in \mathcal{F}(G)$ and $B \in \mathcal{F}(H)$. Since $\mathcal{F}(G)$ and $\mathcal{F}(H)$ are cross-intersecting, we have $A \cap B \neq \emptyset$.
		However, this leads to a contradiction, since by the construction of the sets $A$ and $B$, we clearly have $A \cap B =\emptyset$. Thus, the above assumption is false. Therefore, there exists
		\(1\le \ell\le n\) such that $\mu_{G}(\ell)+\mu_{H}(\ell) > \ell$ as desired.
		
		$(iii) \Rightarrow (iv)$ Assume that there is $\ell$ such that $\mu_{G}(\ell)+\mu_{H}(\ell) > \ell$. We choose $\ell$ to be the smallest such integer. If $\ell = 1$, then $g_{1} = h_{1} = 1$, and we choose $p = q = 1$. Consider $\ell > 1$. Since $\ell$ is chosen as the smallest such number, we have 
		\[
		\mu_{G}(\ell-1) + \mu_{H}(\ell-1) \le \ell - 1 < \ell < \mu_{G}(\ell) + \mu_{H}(\ell).
		\]
		Furthermore, note that 
		\[
		\mu_{G}(\ell) - \mu_{G}(\ell-1) \in \{0,1\}, \qquad 
		\mu_{H}(\ell) - \mu_{H}(\ell-1) \in \{0,1\}.
		\]
		From these observations, we derive the equalities
		\[
		\mu_{G}(\ell) = \mu_{G}(\ell-1) + 1,\qquad \mu_{H}(\ell)= \mu_{H}(\ell-1) + 1.
		\]
		\[
		\mu_{G}(\ell-1) + \mu_{H}(\ell-1) = \ell - 1,\qquad
		\mu_{G}(\ell) + \mu_{H}(\ell)= \ell + 1.
		\]
		Noting that 
		\[
		\mu_{G}(\ell) + \mu_{H}(\ell) = \mu_{G}(\ell-1) + \mu_{H}(\ell-1) + 2,
		\]
		we conclude that $\ell \in G \cap H$. Thus, there exist indices $p, q$ such that $g_{p} = h_{q} = \ell$. Since $\mu_{G}(\ell) = p$ and $\mu_{H}(\ell) = q$, we obtain $\ell = p + q - 1$ and hence $g_{p} = h_{q} = p + q - 1$.
		
		$(iv) \Rightarrow (i)$
		
		Assume that $1 \le p \le r$ and $1 \le q \le s$ such that $g_p=h_q=p+q-1$. Take any $G' \leq G$ and $H' \leq H$. We have $\pi_p(G') \subseteq [p+q-1]$ and $\pi_q(H') \subseteq [p+q-1]$. Since $|\pi_p(G')|+|\pi_q(H')| = p+q > [p+q-1]| $, we have $\pi_p(G')\cap \pi_q(H') \neq \emptyset$. Hence, $G' \cap H' \neq \emptyset$. 
	\end{proof}
	
	\subsection{Proof of Theorem \ref{thm1}}
	
	\begin{proof} Suppose $\mathcal{F}(\mathcal{G})$ is intersecting. We take any $G,H \in \mathcal{G}$. We have $\mathcal{F}(G)$ and $\mathcal{F}(H)$ are cross-intersecting. Lemma \ref{main_lem} gives us that $G \sim_{si} H$. It implies that $\mathcal{G}$ is strongly intersecting. Conversely, suppose that $\mathcal{G}$ is strongly intersecting. We prove that $\mathcal{F}(\mathcal{G})$ is LCIF. Take $A,B \in \mathcal{F}(\mathcal{G})$. Then, there are $G,H \in \mathcal{G}$ such that $A \preceq G$ and $B \preceq H$. We have $\pi_{|G|}(A) \leq G$ and $\pi_{|H|}(B) \leq H$. From $G \sim_{si} H$, we conclude that $\pi_{|G|}(A)\cap \pi_{|H|}(B) \neq \emptyset$. Since $\pi_{|G|}(A) \subseteq A$ and $\pi_{|H|}(B) \subseteq B$, we have $A \cap B \neq \emptyset$.
	\end{proof}
	
	Now, let \(\mathcal{A} \subseteq \binom{[n]}{k}\) be a left-compressed family. Then $\mathcal{A}$ can be described in terms of its maximal elements. Let $\mathcal{G}$ be the collection of all maximal elements of $\mathcal{A}$. Note that each $G \in \mathcal{G}$ is a $k$-set. For any $A \in \mathcal{A}$, there exists $G \in \mathcal{G}$ such that $A \leq G$. This implies $A \preceq G$ and so $A \in  \mathcal{F}(\mathcal{G})$. Conversely, take any $A \in \mathcal{F}(\mathcal{G})$. We have $A \preceq G$ for some $G \in \mathcal{G}$. Since both $A$ and $G$ are $k$-sets, it follows that $A \leq G$. Since $\mathcal{A}$ is left-compressed and $G \in \mathcal{A}$, we have $A \in \mathcal{A}$. Hence, $\mathcal{A}=\mathcal{F}(\mathcal{G})$. We use this result together with Theorem \ref{thm1} to give another proof of the intersecting condition for left-compressed families established by Bond \cite{BB} 
	
	\begin{prop}(\textbf{Bond \cite{BB} Proposition 7}) \label{Bond}.Let $\mathcal{A} \subseteq \binom{[n]}{k}$ be a left-compressed family. Then, $\mathcal{A}$ is intersecting if and only if for any $A=\{a_{1},\ldots,a_{k}\}$ and $B=\{b_{1},\ldots,b_{k}\}$ in $\mathcal{A}$, there exist $i,j$ with $1 \le i,j \le k$ such that $ i+j > \max\{a_i,b_j\}$.
	\end{prop}
	\begin{proof} Since \(\mathcal{A}\) is left-compressed, by the above observation we have \(\mathcal{A}=\mathcal{F}(\mathcal{G})\), where \(\mathcal{G}\) is the family of all maximal elements of \(\mathcal{A}\). 
		
		Assume that $\mathcal{A}$ is intersecting. Take any $A,B \in \mathcal{A}$, then $A \leq G$ and $B \leq H$ for some $G,H \in \mathcal{G}$. Since $\mathcal{A}$ is intersecting, it follows from Theorem \ref{thm1} that \(\mathcal{G}\) is strongly intersecting. So, there exists $\ell$ such that $\mu_{G}(\ell)+\mu_{H}(\ell) >\ell$. From Proposition \ref{tc}, we have $\mu_{A}(\ell) \geq \mu_{G}(\ell)$ and $\mu_{B}(\ell) \geq \mu_{H}(\ell)$. Therefore $\mu_{A}(\ell)+\mu_{B}(\ell) > \ell$. Let \(i=\mu_A(\ell)\) and \(j=\mu_B(\ell)\). Then
		\(1\le i,j\le k\), and we have \(a_i\le \ell\) and \(b_j\le \ell\).
		Hence $i+j>\ell\ge \max\{a_i,b_j\}$.
		
		For converse direction, suppose that the stated condition holds. Take $A,B \in \mathcal{A}$, then there exists $1\le i,j\le k$ such that $i+j > \max\{a_i,b_j\}$. We may assume that $a_i \le b_j$. Then $i+j > b_j$. Moreover, $\{a_1,\ldots,a_i\} \subseteq [b_j]$ and $\{b_1,\ldots,b_j\} \subseteq [b_j]$. So, we must have $A \cap B \neq \emptyset$. We conclude that $\mathcal{A}$ is intersecting.
		
	\end{proof}

	\section{Maximal left compressed intersecting families}
	\subsection{Preliminaries}
	Theorem \ref{thm1} establishes that $\mathcal{F}(\mathcal{G})$ is LCIF if and only if $\mathcal{G}$ is strongly intersecting. In the case where $\mathcal{F}(\mathcal{G})$ is MLCIF, we have the following propositions.
	
	\begin{proposition} \label{tc2} Suppose $\mathcal{G} \subseteq 2^{[n]}$ and 
		$\mathcal{F}(\mathcal{G})$ is an MLCIF. Then:
		\begin{enumerate}[(i)]
			\item $|G \cap [\,k + |G| - 1\,]| = |G|$ for every $G \in \mathcal{G}$;
			\item $\mathcal{G} \cap \binom{[n]}{r}$ is left-compressed for all $1 \le r \le k$;
			\item $|\mathcal{G}| \le \binom{2k - 1}{k - 1}$;
		\end{enumerate}
	\end{proposition}
	\begin{proof} From Theorem \ref{thm1} we have $\mathcal{G}$ is a strongly intersecting family.
		
		(i) Suppose, to the contrary, that there exists $G \in \mathcal{G}$ such that $|G \cap [k+|G|-1]|<|G|$. Suppose that $G=\{g_1,\ldots,g_r\}$. We observe that not every element of $G$ is less than $k+r-1$. So, we have $g_{r} \geq k+r$. For all $H \in \mathcal{G}$, we have $G \sim_{si} H$. From Lemma \ref{main_lem}, there exists $1 \le p \le r$ and $1 \le q \le |H|$ such that $g_p=h_q=p+q-1$. If $p=r$ then $p+q-1=g_p=g_r \ge k+r=k+p$. It implies that $k \le q-1$. This is a contradiction. So, we must have $p<r$. From this, we obtain $g_p\le g_{r-1}$. Thus, $g_p \in \{g_1,\ldots,g_{r-1}\}=G'$. Hence, we also obtain $G' \sim_{si}H$. Thus, the family $\mathcal{G'}=(\mathcal{G}\setminus {G})\cup \{G'\}$ is strongly intersecting. This implies $\mathcal{F}(\mathcal{G}')$ is LCIF. On the other hand, from $G' \subsetneq G$ we get $G \preceq G'$ and therefore $\mathcal{F}(G) \subsetneq \mathcal{F}(G')$. This result contradicts the fact that $\mathcal{F}(\mathcal{G})$ is maximal. Thus, $g_{r} \leq k+r-1$ or $|G \cap [k+|G|-1]|=|G|$.
		
		(ii) Take any $G \in \mathcal{G} \cap \binom{[n]}{r}$ and any $H \le G$. We show that $H \in \mathcal{G}$. For all $K \in \mathcal{G}$, we have $G \sim_{si}K$. Hence, there exists $1 \le \ell \le n$ such that $\mu_{G}(\ell)+\mu_{K}(\ell) > \ell$.  Since $H \le G$, the Proposition \ref{tc} gives that $\mu_{H}(\ell) \ge \mu_{G}(\ell)$. This implies that $\mu_{H}(\ell)+\mu_{K}(\ell) \ge \mu_{G}(\ell)+\mu_{K}(\ell) > \ell$. Thus, $H \sim_{si} K$ for all $K \in \mathcal{G}$. Since $\mathcal{F}(\mathcal{G})$ is maximal, we have $H \in \mathcal{G}$. Since $H \le G$, we have $|H|=|G|=r$. So, $H \in \mathcal{G} \cap \binom{[n]}{r}$.
		
		(iii) Consider $G \in \mathcal{G}$ with $|G|=r$. From (i), we have $g_{r} \leq k+r-1 \leq 2k-1$.  We consider the set $\overline{G}=[k+r] \setminus G$. We have $|\overline{G}|=k$ and $\overline{G}\cap G =\emptyset$. We take $T=G \cup T'$, where $|T'|=k-r$ and $T' \subseteq [k+r+1,n]$ (we can choose such a set $T'$ because $n-k-r \geq k-r$. We have $T \preceq G$, so $T \in \mathcal{A}$. But, it is easy that $\overline{G} \cap T = \emptyset$. Therefore, $\overline{G}\notin \mathcal{A}$.
		
		We first show that the map \(G\mapsto \overline{G}\) is injective.
		Suppose that \(\overline{G}=\overline{H}\) that is $[k+r]\setminus G=[k+s]\setminus H$.  Let \(|G|=r\) and
		\(|H|=s\). From (i), we have $g_r \le k+r-1$ and $h_s \le k+s-1$. We assume that \(r < s\).  We take an element $k+i$ with $r+1 \le i \le s$. Since $k+i \notin [k+r]\setminus G$, we have $k+i \notin [k+s]\setminus H$. This means that $k+i \in H$. In particular, \(k+s\in H\), which contradicts the fact that \(H\subseteq [k+s-1]\). Thus we must have $r=s$. Then, $[k+r]\setminus G=[k+r]\setminus H$ and so $G=H$.
		
		Next, we show that the family
		\[
		\overline{\mathcal{G}}=\{\overline{G}:G\in\mathcal{G}\}
		\]
		is intersecting. Take distinct \(G,H\in\mathcal{G}\). Since
		\(\mathcal{G}\) is strongly intersecting, we have \(G\sim_{si}H\), and
		hence \(G\cap H\neq\emptyset\). Choose \(a\in G\cap H\). Then
		\(a\notin \overline{G}\cup \overline{H}\). Since $
		\overline{G},\overline{H}\subseteq [2k]$. It follows that $|\overline{G}\cup \overline{H}|<2k$. As \(|\overline{G}|=|\overline{H}|=k\), we conclude that $\overline{G}\cap \overline{H}\neq\emptyset$.
		Thus \(\overline{\mathcal{G}}\) is a \(k\)-uniform intersecting family on
		\([2k]\). By the Erdős--Ko--Rado theorem and the injectivity of the map
		\(G\mapsto\overline{G}\), we obtain
		\[
		|\mathcal{G}|\le|\overline{\mathcal{G}}|
		\le \binom{2k-1}{k-1}.
		\]
		
	\end{proof}
	
	Next, we need several lemmas concerning the sizes of
	\(\mathcal{L}(G)\) and \(\mathcal{F}(G)\). First, we recall a general
	formula for the cardinality of the left-generated family
	\(\mathcal{L}(G)\).
	
	\begin{lem}\label{sizel}
		Let \(G=\{g_1,g_2,\ldots,g_r\}\subseteq [n]\). Then
		\[
		|\mathcal{L}(G)|
		=
		\det\left(\binom{g_j-j+1}{i-j+1}\right)_{1\le i,j\le r},
		\]
		where, by convention, \(\binom{x}{m}=0\) if \(m<0\) or \(m>x\).
	\end{lem}
	
	\begin{proof}
		Put \(b_i=g_i-i\) for \(1\le i\le r\). Since
		\(g_1<g_2<\cdots<g_r\), we have
		\[
		0\le b_1\le b_2\le\cdots\le b_r.
		\]
		For any \(S=\{s_1,s_2,\ldots,s_r\}\in\mathcal{L}(G)\), define
		\[
		x_i=s_i-i \qquad (1\le i\le r).
		\]
		Then the condition \(S\le G\) is equivalent to
		\[
		0\le x_1\le x_2\le\cdots\le x_r,\qquad x_i\le b_i
		\quad (1\le i\le r).
		\]
		Conversely, every sequence satisfying these inequalities gives an
		element
		\[
		S=\{x_1+1,x_2+2,\ldots,x_r+r\}\in\mathcal{L}(G).
		\]
		Thus \(|\mathcal{L}(G)|\) is equal to the number of such bounded
		nondecreasing sequences. By the standard determinantal formula for
		bounded nondecreasing sequences, equivalently by the LGV lemma
		\cite{GV}, this number is
		\[
		\det\left(\binom{b_j+1}{i-j+1}\right)_{1\le i,j\le r}.
		\]
		Since \(b_j+1=g_j-j+1\), the desired formula follows.
	\end{proof}
	
	\textbf{Remark} In special cases where \(G\) has a simple structure, one can derive
	explicit formulas for \(|\mathcal{L}(G)|\) by direct counting.
	
	\begin{itemize}
		\item If \(1\le a<b\), then
		\[
		|\mathcal{L}([a,b])|=\binom{b}{a-1}.
		\]
		
		\item If \(1\le a<b\) and \(c\ge 1\), then
		\[
		\left|\mathcal{L}\left(\{a\}\cup [b,b+c]\right)\right|
		=
		\binom{b+c}{c+2}
		-
		\binom{b+c-a}{c+2}.
		\]
		
		\item For \(C_n=\{2,4,\ldots,2n\}\), we have
		\[
		|\mathcal{L}(C_n)|
		=
		\frac{1}{n+2}\binom{2n+2}{n+1},
		\]
		which is the \((n+1)\)-st Catalan number.
		
		\item If \(1\le a<b<c<d\le n\), then
		\[
		|\mathcal{L}([a,b]\cup[c,d])|
		=
		\sum_{m=0}^{a-1}
		\binom{b-a+m}{b-a}
		\binom{d-b+a-1-m}{d-c+1}.
		\]
		
		\item If \(1\le a<b<c\) and \(d\ge 0\), then
		\[
		\begin{aligned}
			&\left|\mathcal{L}\left([a,b]\cup \{c,c+2,\ldots,c+2d\}\right)\right| \\
			&\qquad =
			\sum_{m=0}^{a-1}
			\binom{b-a+m}{b-a}
			\left[
			\binom{c-b+a+2d-m}{d+1}
			-
			\binom{c-b+a+2d-m}{d}
			\right].
		\end{aligned}
		\]
	\end{itemize}
	
	Next, we find a formula for $|\mathcal{F}(G)|$ for
	$G=\{g_1,\ldots,g_r\}\in 2^{[n]}$ with $|G|=r$.
	\begin{lem}\label{sizef}
		Let $G \in 2^{[n]}$ with $|G|\le k$. Then
		\[
		|\mathcal{F}(G)|
		= \sum_{A \in \mathcal{L}(G)}
		\binom{n-\max A}{k-|G|}.
		\]
		
	\end{lem}
	
	\begin{proof} We partition $\mathcal{F}(G)$ as follows. Two sets $S,T \in \mathcal{F}(G)$ belong to the same class if $\pi_r(S)=\pi_r(T)$. Then $\mathcal{F}(G)$ is partitioned into disjoint classes.
		Every $S \in \mathcal{F}(G)$ is constructed as follows. First, take $A \in \mathcal{L}(G)$ ( there are $|\mathcal{L}(G)|$ choices). Next, take $B \in \binom{[\max A+1,n]}{k-|G|}$ ( in $\binom{[n-\max A]}{k-|G|}$ ways). Then, $S=A \cup B \in \mathcal{F}(G)$. We obtain the desired result.
	\end{proof}
	
	For each $G \in 2^{[n]}$ and $\max G \le m \le n$, we define $\mathcal{R}(G,m)$ to be the collection of all sets $S \subseteq [m]$ with the same cardinality
	as \(G\) such that $G \le S$. We also define \(\mathcal{R}_1(G,m)\) to be the collection of all sets
	\(S\in\mathcal{R}(G,m)\) that contain \(1\).
	
	\begin{lem} \label{sizer} Let $G=\{g_1,\ldots,g_r\} \in 2^{[n]}$ and $\max G \le m \le n$. Then,
		$|\mathcal{R}(G,m)|=|\mathcal{L}(m+1-G)|$. Moreover, if $g_1=1$ then $|\mathcal{R}_1(G,m)|=|\mathcal{L}(\pi_{r-1}(m+1-G))|$
	\end{lem}
	\begin{proof} We establish a bijection between $\mathcal{R}(G,m)$ and $\mathcal{L}(m+1-G)$ as follows. For each $S=\{s_1,\ldots,s_r\} \in \mathcal{R}(G)$, we set $T=m+1-S$. Suppose $T=\{t_1,\ldots,t_r\}$. We have $t_i=m+1-s_{r+1-i}$. Since $S \ge G$, $t_i \le m+1-g_{r+1-i}$. So $T \in \mathcal{L}(m+1-G)$. It is easy to see that this assignment establishes a bijection between $\mathcal{R}(G,m)$ and $\mathcal{L}(m+1-G)$. From this, we obtain the desired result.
		
		If \(g_1=1\), the same map restricts to a bijection between
		\(\mathcal{R}_1(G,m)\) and $\mathcal{L}\bigl(\pi_{r-1}(m+1-G)\bigr)$.
	\end{proof}
	
	Let $G=\{g_1,\ldots,g_r\}$ be a generator of size $r$, recall that $\mathcal{H}(G)=\{G_i: 1\le i\le r\}$, where $G_i=\{1\}\cup [i+1,g_i]$ and $\overline{G}=[k+r] \setminus G=\{\overline{g}_1,\ldots,\overline{g}_k\}$. We denote $\ell_i=g_i-i+1=|G_i|$ and we see that $1\le \ell_1 \le \ell_2\le \ldots,\ell_r\le k$. It is obvious that $A \in \mathcal{F}(\mathcal{H}(G))$ if $a_{\ell_i} \le g_i$ for some $1 \le i \le r$ and $A \notin \mathcal{F}(\mathcal{H}(G))$ if $a_{\ell_i} \ge g_i+1$ for all $1 \le i \le r$. The following lemma gives a necessary and sufficient condition for a \(k\)-set containing \(1\) to lie in $\mathcal{F}(\mathcal{H}(G))$.
	
	\begin{lem} \label{lm1} Let $A \in \binom{[n]}{k}$ with $1 \in A$. Then, $A \in \mathcal{F}(\mathcal{H}(G))$ if and only if $A \notin \mathcal{R}_1(\overline{G})$.
	\end{lem}
	\begin{proof} We prove that $A \notin \mathcal{F}(\mathcal{H}(G))$ if and only if $A \in \mathcal{R}_1(\overline{G})$.
		
		Suppose that $A \notin \mathcal{F}(\mathcal{H}(G))$. Then, for all $1 \le i \le r$ we have $a_{\ell_i} \ge g_i+1$. We need to prove that $\overline{G}\le A$. Consider the sequence of indices $\ell_1\le \ell_2 \le \ldots \le \ell_r (*)$. For any $1 \le i \le k$, we will prove that $a_i \ge \overline{g}_i$. there are four cases:
		
		Case $1$. $1 \le i < \ell_1$. Note that $i<\ell_1=g_1$. So
		$[g_{1}-1] \subseteq \overline{G}$. This implies that $a_i \ge i = \overline{g}_i$.
		
		Case $2$. The index $i$ appears in the (*). Let $p$ be the largest index such that $\ell_p=i$. First, consider $p=r$. The set $[g_r]$ has exactly $r$ elements belonging to $G$, so has exactly $(g_r-r)$ elements belonging to $\overline{G}$. Since $g_r+1 \notin G$, the element $g_r+1 \in \overline{G}$. So $\overline{g}_{\ell_r} =g_r+1$. It follows that $a_i=a_{\ell_r} \ge g_r+1=\overline{g}_{\ell_r}=\overline{g}_i$. Next, we consider $p < r$. We have $a_i=a_{\ell_p} \ge g_p+1$. If $g_{p}+1=g_{p+1}$ then $\ell_{p+1}=g_{p+1}-(p+1)+1=g_{p}+1-p-1+1=\ell_p=i$. This contradicts the maximality of $p$. Thus, $g_p+1 \notin G$. So, $g_p+1 \in \overline{G}$. The set $[g_p]$ contains $p$ elements of $G$, namely $g_1,\ldots,g_p$. So, $[g_p]$ contains $g_{p}-p=\ell_{p}-1$ elements of $\overline{G}$. This means that the $\ell_p$-th element of $\overline{G}$ is $g_p+1$. Hence, $a_{i}=a_{\ell_{p}}\ge g_{p}+1= \overline {g}_{\ell_p}=\overline{g}_i$.
		
		Case $3$. There exists $p$ such that $\ell_p < i < \ell_{p+1}$. We have $a_{\ell_{p}} \ge g_{p}+1$. This implies that $a_i \ge i+p$. The set $[g_p]$ contains $p$ elements of $G$, Therefore, it contains \(g_p-p\) elements of $\overline{G}$. On the other hand, from  $\ell_p < i < \ell_{p+1}$, we obtain that $g_p+1 < i+p < g_{p+1}$. It is easy to see that $g_p+1 \in \overline{G}$ and so $[g_p+1,g_{p+1}-1] \subseteq \overline{G}$. In summary, we conclude that the set $[g_{p+1}]$ contains exactly $g_p-p$ elements from $\overline{G} \cap [g_p]$ and contains exactly $g_{p+1}-g_p-1$ elements from $\overline{G} \cap [g_p+1,g_{p+1}-1]$. Thus, there are exactly $(g_p-p)+(g_{p+1}-g_p-1)=g_{p+1}-p-1=\ell_{p+1}-1$ elements belonging to $\overline{G} \cap [g_{p+1}]$. Therefore, we have $\overline{g}_{\ell_{p+1}-1} \le g_{p+1}-1$, equivalently, $\overline{g}_{\ell_{p+1}-1} \le (\ell_{p+1}-1) +p$. Since $i <\ell_{p+1}$, we get $\overline{g}_i \le p+i \le a_i$. 
		
		Case $4$. $\ell_r < i \le k$. We have $[g_r+1,r+k] \subseteq \overline{G}$. Moreover, the set $[g_r]$ contains exactly $r$ elements of $G$ so $[g_r]$ contains exactly $g_r-r=\ell_r-1$ elements of $\overline{G}$. Therefore $\overline{g}_{\ell_r} = g_r+1$ and from this we obtain $\{\overline{g}_{\ell_r},\ldots,\overline{g}_k\}=[g_r+1,r+k]$. Thus, $\overline{g}_i=r+i$ for all $\ell_r <i \le k$. Now, from  $a_{\ell_r} \ge g_r+1=\ell_r+r$, we hvae $a_i \ge r+i=\overline{g}_i$ for all $\ell_r < i \le k$.
		
		Thus, in all cases, we have $\overline{g}_i \le a_i$. Hence, we conclude that $\overline{G} \le A$ that is $A \in \mathcal{R}_1(\overline{G})$.
		
		Conversely, take any $A \in \mathcal{R}_1(\overline{G})$. Then, we have $1 \in A$ and $\overline{G}\le A$. If $A \in \mathcal{F}(\mathcal{H}(G))$ then $\overline{G} \in \mathcal{F}(\mathcal{H}(G))$ by the left-compressed property. Therefore, there exists $1 \le i \le r$ such that $\overline{g}_{\ell_i} \le g_i$. On the other hand, the set $[g_i]$ contains $i$ elements of $G$, so it contains $g_i-i=\ell_i-1$ elements of $\overline{G}$. Thus,  $\overline{g}_{\ell_i} \ge g_i+1$ which is a contradiction. Hence, $A \notin \mathcal{F}(\mathcal{H}(G))$.
	\end{proof}
	
	Finally, we present a lemma concerning the comparison of binomial coefficients, whose proof is simple.
	
	\begin{lem} \label{comp} Suppose $0 \le a \le b \le n$ and $a+b \le n$. Then
		\[
		\binom{n}{a}\le\binom{n}{b}
		\]
	\end{lem}
	
	\begin{proof}
		Since $a\le b$ and $a+b\le n$, we have $
		b\le n-a$.
		
		For every integer $i$ with $a\le i< b$, we have $i+1\le b\le n-a\le n-i$.
		Therefore,
		\[
		\frac{\binom{n}{i+1}}{\binom{n}{i}}
		=\frac{n-i}{i+1}\ge 1.
		\]
		It follows that
		\[
		\binom{n}{a}\le \binom{n}{a+1}\le \cdots \le \binom{n}{b}.
		\]
	\end{proof}
	
	\subsection{Proof of Theorem \protect\ref{thm2}} \mbox{}\par
	
	(i) For all $G_{i}=\{1\}\cup[i+1,g_i] \in \mathcal{H}(G)$, since $i+(g_{i}-i+1)-1=g_{i}$. By Lemma  \ref{main_lem}, we have $G \sim_{si}G_i$. It is obvious that $\mathcal{H}(G)$ is strongly intersecting because every set in \(\mathcal{H}(G)\) contains \(1\). Hence, the family $G \cup \mathcal{H}(G)$ is strongly intersecting. By Theorem \ref{thm1}, $\mathcal{F}_{G}=\mathcal{F}(G) \cup \mathcal{F}(\mathcal{H}(G))$ is LCIF.
	
	Next, we prove that $\mathcal{F}_{G}$ is maximal. Assume that $\mathcal{F}_{G}$ is not maximal. Then, there exists a LCIF $\mathcal{A}$ such that $\mathcal{F}_G \subsetneq \mathcal{A}$. We take $A \in \mathcal{A} \setminus \mathcal{F}_G$ and assume that $A=\{a_{1},\ldots,a_{k}\}$.  We consider the two cases for $a_{1}$. First, we consider $a_{1} \geq 2$. Because $A \notin \mathcal{F}_G$, we have $a_{p} \geq g_{p}+1$, for some $1\leq p \leq r$. We take $B=\{2,3,\ldots,p,g_{p}+1,\ldots,k+g_{p}-p+1\}=[2,p]\cup[g_{p}+1,k+g_{p}-p+1]$. It is clear that $|B|=k$ and $B \leq A$. Since $\mathcal{A}$ is left-compressed and $A \in \mathcal{A}$, we have $B \in \mathcal{A}$. On the other hand, the set $[k+g_p-p+2,n]$ has $(n-k-g_{p}+p-1)$ elements and $n-k-g_{p}+p-1 \geq k-g_{p}+p-1$. So,  we can choose $C \subseteq [k+g_p-p+2,n]$ such that $|C|=k-g_{p}+p-1$. Next, choose $D=\{1,p+1,\ldots,g_{p}\}\cup C$. We have $|D|=k$ and $D \in \mathcal{F}(G_p)$. So $D \in \mathcal{F}(\mathcal{H}(G))$. Now, $B,D \in \mathcal{A}$, but $B \cap D =\emptyset$. Therefore, there does not exist a set $A \in \mathcal{A} \setminus \mathcal{F}_G$ such that $a_{1} \geq 2$.
	
	We now turn to the case $a_{1}=1$. We have $A \notin \mathcal{F}(\mathcal{H}(G))$. From Lemma \ref{lm1}, we have $A \in \mathcal{R}_1(\overline{G})$. Thus, $ \overline{G} \le A$. Since $A \in \mathcal{A}$, we obtain $\overline{G} \in \mathcal{A}$. We choose a set $C= G \cup D$, where $D \in \binom{[k+r+1,n]}{k-r}$. We have $|C|=k$ and so $C \in \mathcal{F}(G)$. So $C \in \mathcal{A}$. However, this contradicts the fact that $C \cap \overline{G}=\emptyset$. 
	
	Since both cases regarding $a_{1}$ lead to a contradiction, it follows that the family $\mathcal{F}(G) \cup \mathcal{F}(\mathcal{H}(G))$ is maximal.
	
	We next estimate the size of the family $\mathcal{F}_{G}$. We recall that $ \mathcal{F}_0(G)=\{ A \in \mathcal{F}(G):1 \notin A\}$. Take any $A=\{a_1,\ldots,a_k\} \in \mathcal{F}(G)$ with $a_1=1$. We have $\pi_r(A) \le G$. We take $B=G\cup[n-k+r+1,n]$. We have $|B|=k$ and $B \in \mathcal{F}(G)$.   From Lemma \ref{main_lem}, there exist indices $1 \le p,q \le k$ such that $a_p=b_q=p+q-1$. It is obvious that $a_2 \le q+1$. If $q\ge r+1$ then $b_q=n-k+q$. Then $b_q=n-k+q=p+q-1$. Therefore $n=k+p-1 < 2k$, which is a contradiction. Thus, we must have $q \le r$. From this, we have $b_q=g_q$. Now, $a_{\ell_q}=a_{g_q-q+1}=a_p =b_q=g_q$. So $\pi_{\ell_q}(A) \le \{1\}\cup[q+1,g_q]=G_q$. Thus, $A \in \mathcal{F}(G_q)$. This implies $A \in \mathcal{F}(\mathcal{H}(G))$. Therefore, all sets in $\mathcal{F}(G)$ that contain \(1\) belong to \(\mathcal{F}(\mathcal{H}(\mathcal{G}))\). It follows that \(\mathcal{F}(G)\) is the disjoint union of the families $\mathcal{F}_0(G)$ and $\mathcal{F}(\mathcal{H}(G))$. This means that,
	$\mathcal{F}_{G}=\mathcal{F}_{0}(G) \cup \mathcal{F}(\mathcal{H}(G))$ with $\mathcal{F}_0(G) \cap \mathcal{F}(\mathcal{H}(G))= \emptyset$.
	
	From Lemma \ref{sizef} we obtain
	\[
	|\mathcal{F}_0(G)|=\sum_{A \in \mathcal{L}_0(G)} \binom{n-\max A}{k-r}
	\]
	where, $\mathcal{L}_0(G)=\{A \in \mathcal{L}(G): 1 \notin A\}$. From Lemma \ref{sizer} and Lemma \ref{lm1}, we have 
	\[
	|\mathcal{F}(\mathcal{H}(G))|=\binom{n-1}{k-1}-|\mathcal{R}_1(\overline{G})|=\binom{n-1}{k-1}-|\mathcal{L}(\pi_{k-1}(n+1-\overline{G}))|
	\]
	
	Hence,
	\[
	|\mathcal{F}_{G}|=|\mathcal{F}_0(G)|+|\mathcal{F}(\mathcal{H}(G))|=\sum_{A \in \mathcal{L}_0(G)} \binom{n-\max A}{k-r}+\binom{n-1}{k-1}-|\mathcal{L}(\pi_{k-1}(n+1-\overline{G}))|
	\]
	
	\textbf{Example}
	(a) $n=7,k=3,G=\{2,3,5\}, r=3$
	\[
	\mathcal{F}_0(G)=\{\{2,3,4\},\{2,3,5\}\},\quad\overline{G}=\{1,4,6\},\quad n+1-\overline{G}=\{2,4,7\}, \quad |\mathcal{L}(\pi_2(n+1-\overline{G})|=|\mathcal{L}(\{2,4\})|=5
	\]
	Hence, $|\mathcal{F}_G|=2+ \binom{6}{2}-5=12$.
	
	(b) $n=9,k=4, G=\{2,3\}, r=2$.
	\[
	|\mathcal{F}_0(G)|=\binom{6}{2}=15, \quad \overline{G}=\{1,4,5,6\},\quad n+1-\overline{G}=\{4,5,6,9\},\quad |\mathcal{L}(\pi_{3}(n+1-\overline{G})|=|\mathcal{L}(\{4,5,6\}|=20
	\]
	Hence, $|\mathcal{F}_G|=15+\binom{8}{3}-20=51$.
	
	(ii) We set $G'=G \setminus \{g_r\}=\{g_1,\ldots,g_{r-1}\}$. Since $G$ is a weak generator, $t(G)< |G|$. So $G'$ is also generator with $t(G')=t(G)$. As proved in (i), we have $\mathcal{F}_{G} =\mathcal{F}_0(G) \cup \mathcal{F}(\mathcal{H}(G))$ and $\mathcal{F}_{G'} =\mathcal{F}_0(G') \cup \mathcal{F}(\mathcal{H}(G'))$. Therefore, $|\mathcal{F}_{G}| =|\mathcal{F}_0(G)|+ |\mathcal{F}(\mathcal{H}(G))|$ and $|\mathcal{F}_{G'}| =|\mathcal{F}_0(G')|+|\mathcal{F}(\mathcal{H}(G'))|$. From this, we obtain 
	\[
	|\mathcal{F}_{G}|-|\mathcal{F}_{G'}|=|\mathcal{F}_{0}(G)|-|\mathcal{F}_0(G')|+|\mathcal{F}(\mathcal{H}(G))|-|\mathcal{F}(\mathcal{H}(G'))|
	\]
	It is obvious that $\mathcal{F}_0(G) \subsetneq \mathcal{F}_0(G')$ and $\mathcal{F}(\mathcal{H}(G))=\mathcal{F}(\mathcal{H}(G')) \cup \mathcal{F}(G_r)$, 
	
	Take any $A \in \mathcal{F}_0(G')\setminus \mathcal{F}_0(G)$. We have  $\pi_{r-1}(A) \le G'$ and $a_r \ge g_r+1$. It follows that $\pi_{r-1}(A) \in \mathcal{L}_0(G')$ and $\{a_r,\ldots,a_k\} \in \binom{g_r+1,n]}{k-r+1}$. Hence,
	\[
	|\mathcal{F}_0(G')\setminus \mathcal{F}_0(G)|=|\mathcal{L}_0(G')|\binom{n-g_r}{k-r+1}.
	\]

	Next, we count the number of elements in $\mathcal{F}(\mathcal{H}(G)) \setminus \mathcal{F}(\mathcal{H}(G'))$. Take any $A\in \mathcal L_0(G')$, and let $B=[g_r]\setminus A$. Since $|A|=r-1, a_{r-1} \le g_{r-1}<g_r$, we have $|B|=g_r-r+1$ and $b_{g_r-r+1}=g_r$. So, for all $1 \le i \le r-1$, we have $a_i \le g_i$. It is easy to see that $|[g_i] \cap A| \ge i$. Thus, $|[g_i]\cap B| \le g_i-i$. This implies that $b_{g_i-i+1} \ge g_i+1$. Therefore $B \notin \mathcal{F}(G_i)$ and $B \in \mathcal{F}(G_r)$. It follows that $B \in \mathcal{F}(\mathcal{H}(G)) \setminus \mathcal{F}(\mathcal{H}(G'))$. Conversely, for any $S \in \mathcal{F}(\mathcal{H}(G)) \setminus \mathcal{F}(\mathcal{H}(G'))$, we have $s_{g_i-i+1} \ge g_i+1$ for all $1 \le i \le r-1$ and $s_{g_r-r+1} \le g_r$. We define $A=[g_r]\setminus \{s_1,s_2,\ldots,s_{g_r-r+1}\}$. We have $|A|=r-1$. On the other hand, $|[g_i] \cap S|\le g_i-i$. So, $|[g_i]\cap A \ge i $. It follows that $a_i \le g_i$ for all $1 \le i \le r-1$. We conclude that $A \le G'$, equivalently $A \in  \mathcal{L}_0(G')$. Recall that, $\ell_i=g_i-i+1$ for each $i \in [r-1]$. Then, $\ell_1 \le \ell_2\le \ldots \le \ell_r$ and $s_{\ell_i} \ge g_i+1$ for all $1 \le i \le r-1$.
	Therefore, $g_{r-1}+1 \le s_{\ell_r-1} \le s_{\ell_r} \le g_r$. Since $\ell_r-\ell_{r-1}=g_r-r+1-g_{r-1}+(r-1)-1=g_r-g_{r-1}-1$, we obtain $s_{\ell_r}=g_r$. From this, we have $s_{\ell_r+1} \ge g_r+1$. Thus, $\{s_{\ell_r+1},\ldots,s_k\} \in \binom{[g_r+1,n]}{k-g_r+r-1}$. In summary, we proved that
	\[
	|\mathcal{F}(\mathcal{H}(G))\setminus \mathcal{F}(\mathcal{H}(G'))|=|\mathcal{L}_0(G')\binom{n-g_r}{k-g_r+r-1}
	\]
	
	Since $\mathcal{F}(\mathcal{H}(G')) \subseteq \mathcal{F}(\mathcal{H}(G))$ and $\mathcal{F}_0(G) \subseteq \mathcal{F}_0(G')$, we have 
	\[
	|\mathcal{F}_0(G') \setminus \mathcal{F}_0(G)|=|\mathcal{F}_0(G')|-|\mathcal{F}_0(G)|
	\]
	and 
	\[
	|\mathcal{F}(\mathcal{H}(G)) \setminus \mathcal{F}(\mathcal{H}(G'))|=|\mathcal{F}(\mathcal{H}(G))|-|\mathcal{F}(\mathcal{H}(G'))|
	\]
	Therefore 
	\[
	\begin{aligned}
		|\mathcal{F}_{G'}| - |\mathcal{F}_{G}|
		&= |\mathcal{F}_0(G')| + |\mathcal{F}(\mathcal{H}(G'))|
		- |\mathcal{F}_0(G)| - |\mathcal{F}(\mathcal{H}(G))| \\
		&= |\mathcal{L}_0(G')|\binom{n-g_r}{k-r+1}
		- |\mathcal{L}_0(G')|\binom{n-g_r}{k-g_r+r-1} \\
		&= |\mathcal{L}_0(G')|
		\left(
		\binom{n-g_r}{k-r+1}
		-
		\binom{n-g_r}{k-g_r+r-1}
		\right).
	\end{aligned}
	\]
	Now, note that $(k-r+1)+(k-g_r+r-1)=2k-g_r<n$. Moreover, $k-g_r+r-1 < n-g_r$ and since $g_r \le k+r-1$, we also have $n-g_r \ge n-k-r+1 \ge k-r+1$. Moreover, $(k-r+1)-(k-g_r+r-1)=g_r-2r+2$. So, the remaining part of the theorem follows directly from Lemma \ref{comp}.
	
	(iii) Suppose that $G$ is a strong generator with $t(G)=|G|=s$. If $s=2$ then $G=\{2,3\}$. We have $K$ is a generator and $K \le G$. So $K=G$ and the conclusion is immediate. We consider $s \ge 3$. Then, $G=\{g_1,\ldots,g_{s-2},2s-2,2s-1\}$, where $g_i\ge 2i$ for all $1 \le i \le s-2$. Since $K \le G$ is a generator, from (i) we have $\mathcal{F}_{K}=\mathcal{F}_0(K) \cup \mathcal{F}(\mathcal{H}(K))$.
	Assume $K=\{k_1,\ldots,k_s\} \in \mathcal{L}_0(G)$, We will partition the set $\mathcal{F}_0(G) \setminus \mathcal{F}_0(K)$ into disjoint classes as follows. Two sets $A, B \in \mathcal{F}_0(G) \setminus \mathcal{F}_0(K)$ belong to the same class if and only if $\pi_s(A)=\pi_s(B)$. For $A \in \mathcal{L}_0(G) \setminus \mathcal{L}_0(K)$, we define $\mathcal{F}_A$ to be the sets of all $S \in \mathcal{F}_0(G) \setminus \mathcal{F}_0(K))$ such that $\pi_s(S)=A$. Thus, we have 
	\[
	\mathcal{F}_0(G) \setminus \mathcal{F}_0(K)=\bigcup_{A \in \mathcal{L}_0(G)\setminus \mathcal{L}_0(K)}\mathcal{F}_{A}
	\]
	Moreover, it is easy to see that
	\[
	|\mathcal{F}_A|=\binom{n-\max A}{k-s}
	\]
	
	Consider $\overline{A}=[k+s]\setminus \pi_s(A)$. It is obvious that $|\overline{A}|=k$. Note that $\overline{a}_1=1$ and $a_{s}-s \le g_{s}-s \le k-1$. We define 
	
	\[
	\mathcal{G}_A =
	\left\{
	S : S = \pi_{a_s-s}(\overline{A}) \cup S'
	\text{ for some } S' \in \binom{[a_s+1,n]}{k-a_s+s}
	\right\}.
	\]
	it is obvious that 
	\[
	|\mathcal{G}_A|=\binom{n-a_s}{k-a_s+s}
	\]
	We have the following facts about $\mathcal{G}(A)$
	
	\textbf{Fact 1} $\mathcal{G}_A \subseteq \mathcal{F}(\mathcal{H}(K))\setminus \mathcal{F}(\mathcal{H}(G))$. 
	
	\textbf{Proof}.Take any $S \in \mathcal{G}_A$. Since $A \in \mathcal{L}_0(G)$ and $ A \notin \mathcal{L}_0(K)$, there exists $1 \le p \le s$ such that $k_p+1 \le a_p \le g_p$. The set $[k_p]$ contains at most $p-1$ elements of $A$. Hence, $[k_p]$ contains at least $k_p-p+1$ elements of $\overline{A}$. It implies that $\overline{a}_{k_p-p+1}\le k_p$. Now, we note that $k_p -p+1 \le a_p-p \le a_s-s$. So, $s_{k_p-p+1}=\overline{a}_{k_p-p+1} \le k_p$. It follows that $S \in \mathcal{F}(K_p)$ and therefore $S \in \mathcal{F}(\mathcal{H}(K))$.
	
	Next, we prove that $s_{g_i-i+1} \ge g_i+1$ for all $1 \le i \le s$. After proving this, then we immediately obtain that $S \notin \mathcal{F}(G_i)$ and so $S \notin \mathcal{F}(\mathcal{H}(G))$. For all $1 \le i \le s$, we have $a_i \le g_i$. Thus, the set $[g_i]$ contains at least $i$ elements of $A$ (namely $a_1,\ldots,a_i$). So the set $[g_{i}]$ contains at most $g_i-i$ elements of $\overline{A}$. From this, we have $\overline{a}_{g_i-i+1} \ge g_i+1$. We set $q=a_s-s$. For those \(i\) such that $g_i-i+1 \le q$, we have $s_{g_i-i+1}=\overline{a}_{g_i-i+1} \ge g_i+1$. In the case of $g_i-i+1 > q$, we see that $\overline{a}_{q+j}=a_s+j$ for all $1 \le j \le k-q$. Hence, 
	\[
	s_{g_i-i+1} \ge \overline{a}_{q+g_i-i+1-q}=a_s +g_i-i+1-q=g_i+1+s-i \ge g_i+1.
	\]
	
	\textbf{Fact 2.} If $A,B\in \mathcal L_0(G)\setminus \mathcal L_0(K)$ and $A\ne B$, then $\mathcal G_A\cap \mathcal G_B=\varnothing$
	
	\textbf{Proof.}
	Take any $A,B\in \mathcal L_0(G)\setminus \mathcal L_0(K)$. We have $|A|=|B|=s$. Assume that 
	\[
	A=\{a_1,\ldots,a_s\}, \quad B=\{b_1,\ldots,b_s\}\
	\]

	Since $|[a_s]\setminus A|=a_s-s$, we have $\pi_{a_s-s}(\overline{A})=[a_s]\setminus A$. Similarly, $\pi_{b_s-s}(\overline{B})=[b_s]\setminus B$. There are two cases.
	
	\textbf{The case $a_s=b_s$}. We have $|\pi_{a_s-s}(\overline{A})|=\pi_{b_s-s}(\overline{B})|$. Since $A \ne B$ and $|A|=|B|=s$, there exists $c \in A \setminus B$. We have $c \notin \overline{A}$. From $c \in A$, we have $c \le a_s=b_s$. So, $c \in [b_s] \setminus B$. Therefore $c \in \pi_{b_s-s}(\overline{B})$. On the other hand, it is obvious that $c \notin \pi_{a_s-s}(\overline{A})$. Thus,  $\pi_{a_s-s}(\overline{A}) \neq \pi_{b_s-s}. (\overline{B})$. Now, take any $S \in \mathcal{G}_A$ and $T \in \mathcal{G}_B$. We have $S=\pi_{a_s-s}(\overline{A})\cup S'$ and $T=\pi_{b_s-s}(\overline{B})\cup T'$, where $S' \in \binom{[a_s+1,n]}{k-a_s+s}$ and $T' \in \binom{[b_s+1,n]}{k-b_s+s}$. It is obvious that $S \neq T$ and so $\mathcal{G}_A \cap \mathcal{G}_B=\emptyset$.   
	
	\textbf{The case $a_s < b_s$}. We have $|\pi_{a_s-s}(\overline{A})|<\pi_{b_s-s}(\overline{B})|$. Since $A \ne B$ and $|A|=|B|=s$, there exists $c \in A \setminus B$. We have $c \le a_s <b_s$. So $c \notin \pi_{a_s-s}(\overline{A})$ and $c \in \pi_{b_s-s}(\overline{B})$. For any $S \in \mathcal{G}_A$ and any $T \in \mathcal{G}_B$, we have $S=\pi_{a_s-s}(\overline{A})\cup S'$ and $T=\pi_{b_s-s}(\overline{B})\cup T'$, where $S' \in \binom{[a_s+1,n]}{k-a_s+s}$ and $T' \in \binom{[b_s+1,n]}{k-b_s+s}$. We always have $c \in T$ since $c \in \pi_{b_s-s}(\overline{B})$. Note that $\min S' \ge a_s+1 >a_s \ge c$. Therefore, we always have $c \notin S $ for every $S'$ and $c \in T$ for every $T'$. Hence, $S \neq T$ and we conclude that $\mathcal{G}_A \cap \mathcal{G}_B=\emptyset$.
	
	\textbf{Fact 3} $|\mathcal{F}_A| <|\mathcal{G}_A| $
	
	\textbf{Proof} Since $\max A=a_s \le g_s=2s-1$, we have $k-s<k-a_s+s$. On the other hand
	$(k-s)+(k-a_s+s)=2k-a_s \le n-a_s$. It follows from Lemma \ref{comp} that
	\[
	|\mathcal{F}_A|=\binom{n-a_s}{k-s} < \binom{n-a_s}{k-a_s+s}=|\mathcal{G}_A|
	\]
	
	Now, we return to the theorem. By using the Facts 1, 2, and 3, we obtain
	\[
	|\mathcal{F}_0(G) \setminus \mathcal{F}_0(K)| < |\mathcal{F}(\mathcal{H}(K)) \setminus \mathcal{F}(\mathcal{H}(G))|
	\]
	which is equivalent to
	\[
	|\mathcal{F}_0(G)|-|\mathcal{F}_0(K)| < |\mathcal{F}(\mathcal{H}(K))|-|\mathcal{F}(\mathcal{H}(G))|
	\]
	This completes the proof.

\end{document}